\documentclass[12pt]{article}
\usepackage[margin=1in]{geometry}
\usepackage{tikz}
\usepackage{amsmath,amssymb,amsthm}
\usepackage[hidelinks]{hyperref}
\usepackage{xcolor}
\usepackage{float}
\usepackage{amssymb}

\newtheorem{theorem}{Theorem}[section]
\newtheorem{proposition}{Proposition}[section]

\newtheorem{example}[theorem]{Example}

\tikzset{>=latex}
\tikzset{->-/.style={decoration={
  markings,
  mark=at position .5 with {\arrow{>}}},postaction={decorate}}}

\theoremstyle{definition}

\newtheorem{definition}{Definition}[section]

\title{Sch\"{u}tte's property for sets of tournaments and an application to dice games}
\author{Joel Jeffries \footnote{Iowa State University, Ames, IA 50011, USA, \texttt{\{joeljef}@iastate.edu\}}}
\date{\empty}

\begin{document}

\maketitle

\begin{abstract}
A tournament has Sch\"{u}tte's property $S_k$ if for every set of $k$ vertices, there is a vertex which dominates the set. In 1963, Erd\H{o}s provided bounds for $f(k)$, the smallest order of an $S_k$ tournament. Sch\"{u}tte's property has various applications, including the design of unfair dice games. A set of dice introduced by James Grime motivates a generalization of Sch\"{u}tte's property to sets of tournaments: a set of tournaments on the same vertex set has property $S_k$ if for every set of $k$ vertices, there is a vertex which dominates the set in at least one of the tournaments. 

We explore this generalization and provide bounds on the fewest number of vertices needed to have an $S_k$ set of $m$ tournaments. We then apply these results to introduce a few new sets of dice similar to Grime's dice that can be used to play a game that gives one player an advantage.
\end{abstract}

\section{Introduction}

A tournament is a directed graph obtained by directing each edge of a complete graph. In a directed graph, a vertex $v$ \emph{dominates} a set of vertices if $v$ is directed towards every vertex in the set. In a 1963, Erd\H{o}s \cite{erdos} introduced Sch\"{u}tte's property for tournaments.

\begin{definition}[Sch\"{u}tte's property]
A tournament $T$ has property $S_k$ if, for every set of $k$ vertices, there is a vertex $v \in V(T)$ that dominates the set. Define $f(k)$ to be the smallest order of an $S_k$ tournament.
\end{definition}

\begin{figure}[htp!] \centering 
\includegraphics[scale=1]{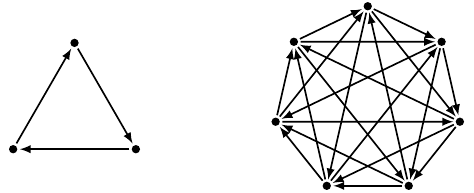}
\caption{An example of a $3$-vertex tournament that is $S_1$ and a $7$-vertex tournament that is $S_2$.}
\label{fig:SkTournaments}
\end{figure} 

In his article, Erd\H os proved the existence of $S_k$ tournaments for all $k$ using the probabilistic method and provided bounds for $f(k)$. In 1965, Szekeres and Szekeres improved the lower bound on $f(k)$ \cite{2Szekeres}.

\begin{theorem}[Erd\H os (1963), Szekeres and Szekeres (1965)]\label{thrm:Erdos}

\begin{enumerate}
    \item[]
    \item $f(k)\leq\min\left\{ n : 2^k \binom{n}{k} (1-2^{-k})^{n-k} < 1 \right\} \leq (\log(2) + O(1)) k^2 2^k$ for large $k$  (Erd\H os)
    \item $f(k) \geq 2^{k+1} - 1$ (Erd\H os)
    \item $f(k) \geq (k+2)2^{k-1} - 1$ for $k > 2$ (Szekeres and Szekeres)
\end{enumerate}
\end{theorem}

The bounds in Theorem \ref{thrm:Erdos} remain the best known bounds for $f(k)$, but there have been several variations on Sch\"{u}tte's property studied over the years. In order to improve the lower bound on $f(k)$, Szekeres and Szekeres considered tournaments where every $k$ set of vertices is dominated by at least $m$ other vertices  \cite{2Szekeres}. They conjectured that their lower bound is the true value for $f(k)$. Borowiecki, Haluszczak, and Tuza explored a bipartite variation of Sch\"{u}tte's property motivated by geometric designs \cite{bipartiteSchutte}. Sch\"{u}tte's property and the bounds for $f(k)$ have some direct ties to other problems about domination in digraphs(\cite{boundeddomination, dominationirredundance}).
While Erd\H os proved the existence of tournaments with property $S_k$ probabilistically, there has been some effort to provide constructions for $S_k$ tournaments (\cite{graham_spencer_construction}, \cite{tyszkiewiczsimple}). However, these constructions result in tournaments that grow faster in size than the asymptotics given by the probabilistic construction of Erd\H os. 

Some values of $f(k)$ are known for small $k$, all achieved by Paley tournaments. The Paley tournament $P_p$ is the tournament on a prime $p$ number of vertices where  $(i,j)$ is an edge if $j - i$ is a quadratic residue modulo $p$. The tournaments $P_3$, $P_7$, and $P_{19}$ are $S_1$, $S_2$, and $S_3$ respectively, each of these having a number of vertices equal to the lower bound given for $f(k)$ in Theorem \ref{thrm:Erdos}. According to Boz\' oki, Fisher showed computationally that $P_{67}$, $P_{331}$, and $P_{1163}$ are the smallest $S_4$, $S_5$, and $S_6$ Paley tournaments respectively \cite{bozoki}.

Tournaments with Sch\"{u}tte's property have applications across varying fields, from voting theory to teaching theory (\cite{voting},\cite{teaching}). 
These tournaments are particularly well suited for the design of unfair games. For example, consider the $7$ dice created by Oskar van Deventer shown in Figure \ref{fig:oskarDice}. Each arrow between two dice in this diagram shows which die is expected to win if the two are rolled against each other, all with probability $\frac{5}{9}$. The tournament generated by these these winning relations is $P_7$, the smallest possible $S_2$ tournament. These dice can be used to play the following 3-player game: two players each select a die and a third player selects a die in response. The first two players each roll off against the last many times. Player 3 wins if they can roll higher than each opponent more often than not.

\begin{figure}[htp!] \centering 
\includegraphics[scale=0.9]{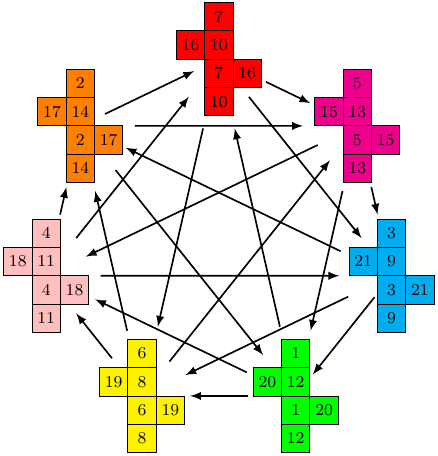}
\caption{Oskar's dice. Each arrow represent a win with probability $\frac{5}{9}$.}
\label{fig:oskarDice}
\end{figure}

Since the tournament formed by Oskar's dice is $S_2$, this game favors player 3. No matter which two dice the first two players choose, player 3 can select a die that is expected to beat each with probability $\frac{5}{9}$. Furthermore, since $f(2) = 7$, there is no $S_2$ tournament on fewer than $7$ vertices, so there is no set of fewer than $7$ dice that has this property of favoring player 3. However, while Oskar's dice have the minimum number of dice needed to play this particular game, a set of $5$ dice created by James Grime can be used to play a similar game. These dice are shown in Figure \ref{fig:grimeDice}. 
In this diagram, each arrow in the left image shows which die is expected to roll higher when the dice are rolled once and compared. Each arrow in the right image shows which die is expected to roll higher when the dice are rolled twice and the summed results are compared. Note that this set of dice is the second set created by Grime, the first set nearly realizing these tournaments, but with smaller numbers on the faces \cite{grime}.

\begin{figure}[htp!] \centering 
\includegraphics[scale=0.9]{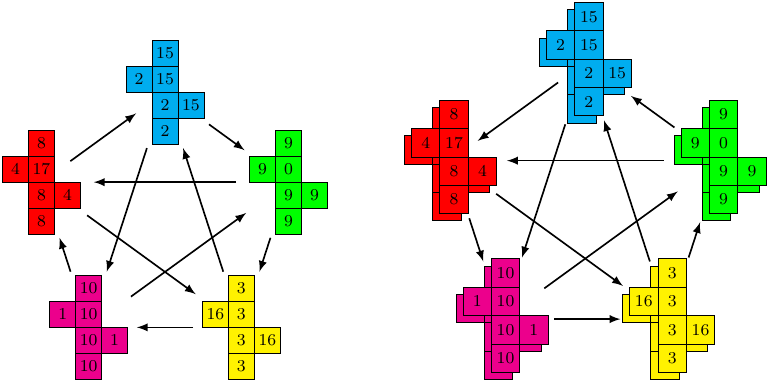}
\caption{Grime's dice.}
\label{fig:grimeDice}
\end{figure}

These dice can be used to play the following 3-player game: two players each select a die. In response, the third player selects a die \emph{and} decides if the game will be played by rolling once or twice. The first two players each roll off against the last many times. Player 3 wins if they can roll higher than each opponent more often than not.
This game again favors player 3, who can always pick a die to be able to beat any two other dice, either by rolling once or by rolling twice (check a few pairs of dice and find a die that is expected to beat them in Figure \ref{fig:grimeDice} to convince yourself).

The two tournaments formed by Grime's dice have a similar property to Sch\"{u}tte's property $S_2$: given any two vertices, there exists a third that dominates them in one tournament \emph{or} the other. However, neither tournament by itself is $S_2$.
This property naturally leads to an extension of Sch\"{u}tte's property to sets of tournaments. In this paper, we formalize this property and discuss bounds on the orders of tournaments that can have this property. Then, we apply these results to construct sets of dice that can be used to play a game similar to the one played with Grime's dice.

\section{Sch\"{u}tte's property for sets of tournaments}\label{sec:modifiedSchuttes}

Throughout this section, we will be considering sets of labeled tournaments on a fixed vertex set. The following definition is a generalization of Sch\"{u}tte's property to  a set of tournaments. If a vertex $v$ dominates a set of vertices $U$, we say $v \rightarrow U$.

\begin{definition}[Sch\"{u}tte's property for sets of tournaments]
Let $\tau = \{T_1, \dotsc , T_m\}$ be a set of $m$ labeled tournaments each on the same vertex set $V(\tau)$. The set $\tau$ has Sch\"{u}tte's property $S_k$ if for every set of  $k$ vertices $U \subset V(\tau)$, there is a tournament $T_i \in \tau$ and a vertex $u \in V(\tau)$ such that $u \rightarrow U$ in $T_i$.
\end{definition}

For brevity, we will call an $S_k$ set of $m$ tournaments on a vertex set of size $n$ an $S_k$ $m$-set of order $n$. For example, the tournaments realized by Grime's dice in Figure \ref{fig:grimeDice} is an $S_2$ $2$-set of order $5$. Since $S_k$ tournaments exist for every $k$ and form a set of size $1$, the first bullet of the following proposition implies the existence of $S_k$ $m$-sets for every $k$ and $m$. 

\begin{proposition}\label{prop:simpleResult}
If there exists an $S_k$ $m$-set of order $n$, then 
\begin{itemize}
    \item there exists an $S_k$ $m'$-set of order $n$ for all $m' \geq m$.
    \item there exists an $S_{k'}$ $m$-set of order $n$ for all $|V(\tau)|-1 \leq k' \leq k$.
\end{itemize}
\end{proposition}

\begin{proof}
Let $\tau = \{T_1, \dotsc , T_m\}$ be an $S_k$ $m$-set of order $n$ and let $V = V(\tau)$. 

Let $m' \geq m$. Form $\tau'$ by adding $m'-m$ arbitrary tournaments on $V$ to $\tau$. Then if $A$ is a $k$ subset of $V$, then since $\tau$ is $S_k$, there exist $T_i \in \tau$ and $u \in V$ such that $u \rightarrow A$ in $T_i$. But $T_i \in \tau'$, so $\tau'$ is an $S_k$ $m'$ set.

Now let $|V(\tau)|-1 \leq k' \leq k$. Let $B$ be a $k'$ subset of $V$. 
Form $B'$ by adding $k - k'$ arbitrary vertices of $V$ to $B$. Then $B'$ is a $k$ subset of $V$, so there exist $T_i \in \tau$ and $u \in V$ such that $u\rightarrow B'$. 
But $B \subseteq B'$, so $u \rightarrow B$. Hence $\tau$ is an $S_{k'}$ $m$-set.
\end{proof}

The second bullet of Proposition \ref{prop:simpleResult} implies that it is easy to create $S_k$ sets of tournaments with a large number of vertices. So, it is more interesting to find $S_k$ sets of tournaments with a small number of vertices.

\begin{definition}
    For $m, k \in \mathbb{N}$, define $f(m,k)$ to be the minimum size of a vertex set $V$ such that there is an $S_k$ set of $m$ tournaments on $V$.
\end{definition}

Proposition \ref{prop:simpleResult} shows that $f(m,k)$ is decreasing in $m$ and increasing in $k$. This means that for fixed $k$, fewer vertices are needed to dominate every $k$ set as more tournaments are allowed. Theorem \ref{thrm:constructionBound} addresses how many fewer vertices are needed, but the following result shows that to obtain an $S_k$ set, no more than $k+1$ tournaments would ever be needed. Note in the following proof that $V/i$ is the set $V$ with the element $i$.

\begin{proposition}\label{prop:m=k+1}
Let $k \in \mathbb{Z}$. Then $f(k+1,k) = k+1$. Furthermore, if $m \geq k+1$, $f(m,k) = k+1$.
\end{proposition}

\begin{proof}
    Note that if $\tau$ is a set of tournaments such that $|V(\tau)| 
    \leq k$, then $\tau$ is not $S_k$, since there is no vertex outside of any $k$ set of vertices of $V(\tau)$, if such a set even possible. So, $f(m,k) \geq k+1$.

    Let $V = \{1, \dotsc, k+1\}$ and define $\tau = \{T_1, \dotsc, T_{k+1}\}$ on $V$ such that $i \rightarrow V/i$ in $T_i$ for all $1 \leq i \leq k+1$. Then every $k$ subset of $V$ is of the form $V/i$, which is dominated by vertex $i$ in $T_i$. Thus $\tau$ is an $S_k$ $(k+1)$-set and $f(k+1,k) \leq k+1$. Hence, $f(k+1,k) = k+1$.

    If $m \geq k+1$, by the first bullet of Proposition \ref{prop:simpleResult}, $f(m,k) \leq f(k+1,k) = k+1$, and so $f(m,k) = k+1$.
\end{proof}

\begin{figure}[htp!] \centering 
\includegraphics[scale=1.2]{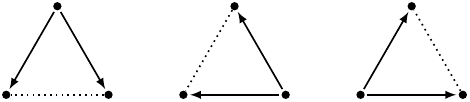}
\caption{Three $S_2$ tournaments achieving the value $f(3,2)=3$.}
\label{fig:f(3,2)}
\end{figure}

The construction in Proposition \ref{prop:heavyConstruction} does the heavy lifting in the proof of Theorem \ref{thrm:constructionBound}, which shows how drastically $f(m,k)$ decreases in $m$. Note that in the following proof, given a digraph $D$ and set of vertices $A$, $D[A]$ is the induced digraph on $A$. Also note that for an integer $n$, $[1,n]$ is the discrete interval $\{1, \dotsc , n\}$.

\begin{proposition}\label{prop:heavyConstruction}
    If there exists an $S_{k_1}$ $m_1$-set of order $n_1$ and an $S_{k_2}$ $m_2$-set of order $n_2$, then there exists an $S_{k_1+k_2+1}$ $(m_1 + m_2)$-set of order $n_1 + n_2$.
\end{proposition}

\begin{proof}
    Let $\tau_1 = \{T_{(1,1)}, \dotsc, T_{(1,m_1)}\}$ be an $S_{k_1}$ $m_1$-set of order $n_1$ on a vertex set $V_1$ and $\tau_2 = \{T_{(2,1)}, \dotsc, T_{(2,m_1)}\}$ be an $S_{k_2}$ $m_2$-set of order $n_2$ on a vertex set $V_2$ (with $V_1 \cap V_2 = \emptyset$). Define the set $\tau' = \{T'_{(1,1)}, \dotsc, T'_{(1,m_1)}, T'_{(2,1)}, \dotsc, T'_{(2,m_1)}\}$ of tournaments as follows:
    \begin{itemize}
        \item $V(\tau') = V' = V_1 \cup V_2$.
        \item For $i \in [1, m_1]$, $T'_{(1,i)}[V_1] = T_{(1,i)}$.
        \item For $i \in [1, m_1]$, if $v_1 \in V_1$ and $v_2 \in V_2$, then $(v_1,v_2) \in E(T'_{(1,i)})$.
        \item For $j \in [1, m_2]$, $T'_{(2,j)}[V_2] = T_{(2,j)}$.
        \item For $j \in [1, m_2]$, if $v_1 \in V_1$ and $v_2 \in V_2$, then $(v_2,v_1) \in E(T'_{(2,j)})$.
    \end{itemize}

    We show $\tau'$ is $S_{k_1 + k_2+1}$. Let $A \subseteq V_1 \cup V_2$ such that $|A| = k_1 + k_2+1$. Then since $V_1 \cap V_2 = \emptyset$, either $A \cap V_1 \leq k_1$ or $A \cap V_2 \leq k_2$.

    Suppose $A \cap V_1 \leq k_1$. Then, since $\tau_1$ is $S_{k_1}$, there exists $i \in [1,m_1]$ and $u \in V_1$ such that $u \rightarrow A \cap V_1$ in $T_{(1,i)}$. 
    Then since $T'_{(1,i)}[V_1] = T_{(1,i)}$, we have that $u \rightarrow A \cap V_1$ in $T'_{(1,i)}$. 
    Furthermore, since $u \in V_1$ and  $(u,v_2) \in E(T'_{(1,i)})$ for all $v_2 \in V_2$, we have that $u \rightarrow A \cap V_2$ in $T'_{(1,i)}$. 
    Hence $u \rightarrow A$ in $T'_{(1,i)}$.

    Similarly, if $A \cap V_2 \leq k_2$, since $\tau_2$ is $S_{k_2}$, there exists $j \in [1, m_2]$ and $u \in V_2$ such that $u \rightarrow A$ in $T'_{(2,j)}$.
    Thus, any $k_1 + k_2+1$ set of $V(\tau')$ is dominated in some tournament of $\tau'$, and so $\tau'$ is an $S_{k_1+k_2+1}$ $(m_1+m_2)$ set of order $n_1 + n_2$.
\end{proof}

\begin{example}\label{example:f(m,k)}
    Figures \ref{fig:f(2,2)} and \ref{fig:f(3,5)} show sets coming from the construction in Proposition \ref{prop:heavyConstruction}. Note that the dashed edges in these tournaments can be directed arbitrarily, and the bold edges show smaller $S_k$ tournaments within the structure.

    Note that the set in Figure \ref{fig:f(3,5)} is formed by iterating the construction twice. The first iteration combines an $S_1$ tournament on vertices $v_1, v_2$, and $v_3$ with an $S_1$ tournament on vertices $v_4$, $v_5$, and $v_6$. The next combines the resulting $S_3$ set of $2$ tournaments with an $S_1$ tournament on vertices $v_7$, $v_8$, and $v_9$. 
\end{example}

\begin{figure}[htp!] \centering 
\includegraphics[scale=1]{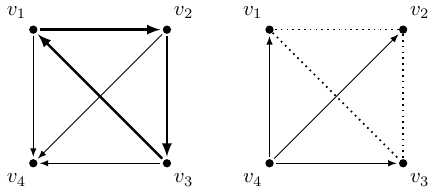}
\caption{Two $S_2$ tournaments achieving the value $f(2,2) = 4$.}
\label{fig:f(2,2)}
\end{figure}


\begin{figure}[H] \centering 
\includegraphics[scale=1]{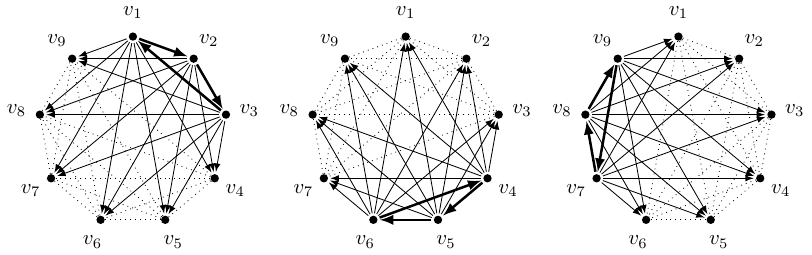}
\caption{Three $S_5$ tournaments achieving the value $f(3,5) = 9$.}
\label{fig:f(3,5)}
\end{figure}

We can use the construction from Proposition \ref{prop:heavyConstruction} to bound the value of $f(m,k)$ recursively. In particular, this allows us to bound $f(m,k)$ by values of $f(k)$.

\begin{theorem}\label{thrm:constructionBound}
For any $m_1 + m_2 = m$ and $k_1 + k_2 = k-1$, $f(m,k) \leq f(m_1,k_1) + f(m_2,k_2)$. 

Additionally, if $k+1 = am + b$ with $1 \leq b \leq m$, then $f(m,k) \leq bf(a) + (m-b)f(a-1)$.
\end{theorem}

\begin{proof}
    The first bound comes directly from the construction from Proposition \ref{prop:heavyConstruction}. We prove the second bound by induction on $m$.

    If $m=1$, then $a = k+1$ and $b = 0$, and so the bound reads $f(1,k)\leq f((k+1)-1) = f(k)$ which is true by definition.
    Assume the bound holds for all $m' \leq m$. Then since $k+1 = am+b$ implies $(a-1) + (a(m-1)+b-1) = k-1$, using the bound above where $m_1 = 1$, $m_2 = m-1$, $k_1 = a-1$, and $k_2 = a(m-1) + b-1$, we have
    \begin{align*}
        f(m,k) &\leq f(1,a-1) + f(m-1, a(m-1)+b-1)\\
        &\leq f(a-1) + bf(a) + (m-1-b)f(a-1)\\
        &= bf(a) + (m-b)f(a-1)
    \end{align*}
    
\noindent where the first line comes from the first bound in the theorem and the second line comes from the inductive assumption.
\end{proof}
Using this result and the bound provided by Erd\H os in Theorem \ref{thrm:Erdos}, we get the following corollary. This shows that for $m > 1$, $f(m,k)$ grows considerably slower than $f(k)$ with respect to $k$. Note that one could follow a similar probabilistic argument as Erd\H{o}s used in Theorem \ref{thrm:Erdos} to obtain a bound for $f(m,k)$, but this will give a larger bound than the one shown here.

\begin{theorem}
For $m, k \in \mathbb{Z}$, $f(m,k) \leq m f\!\left(\left\lceil \frac{k-m+1}{m} \right\rceil \right)$.
In particular, $f(m,k) \leq (\log(2) + O(1)) \frac{k^2}{m} 2^{k/m}$ for fixed $m$ and large $k$.
\end{theorem}

\begin{proof}
    Let $k+1 = am+b$ with $0 \leq b < m$. Then by Theorem \ref{thrm:constructionBound} (and since $f(k)$ is increasing in $k$),
    \begin{align*}
    f(m,k) &\leq bf(a) + (m-b)f(a-1) \\
    &\leq b f(a) + (m-b) f(a) \\
    &= m f(a).
    \end{align*}
    Then since $b < m$, we have that $a = \frac{k+1-b}{m} \leq \left\lceil \frac{k-m+1}{m} \right\rceil$, and so $f(m,k) \leq  m f\!\left(\left\lceil \frac{k-m+1}{m} \right\rceil \right)$.

    Let $k' = \left\lceil \frac{k-m+1}{m} \right\rceil$ and note that $k' \leq \frac{k}{m}$. Then by Theorem \ref{thrm:Erdos},

    \begin{align*}
        f(m,k) &\leq m f(k')\\
        &\leq m \left((\log(2) + O(1)) {k'}^2 2^{k'} \right) \text{ for large } k \\
        &\leq m (\log(2) + O(1)) {\left(\frac{k}{m}\right)}^2
2^{k/m}\\
        &= (\log(2) + O(1)) \frac{k^2}{m} 2^{k/m}.
\end{align*}
\end{proof}

The bounds from Theorem \ref{thrm:constructionBound} are pretty good for small $m$ and $k$. Table \ref{table:smallf(m,k)} shows the upper bounds of $f(m,k)$ for small $m$ and $k$ calculated from that theorem using the known values of $f(k)$. Note that the first column shows the upper bounds for $f(k)$ achieved by Paley graphs mentioned in the introduction. The starred values have been confirmed to be exact values for $f(m,k)$ using a SAT formulation of the problem discussed further in \cite{jeffries2024generalizations}. 

\begin{figure}[H]
\begin{center}
\begin{tabular}{|l|l|l|l|l|l|}
\hline
k\textbackslash{}m & 1 & 2 & 3 & 4 & 5 \\ \hline
0 & 1* &  &  &  &  \\ \hline
1 & 3* & 2* &  &  &  \\ \hline
2 & 7* & 4* & 3* &  &  \\ \hline
3 & 19* & 6* & 5* & 4* &  \\ \hline
4 & 67 & 10* & 7* & 6* & 5* \\ \hline
5 & 331 & 14 & 9* & 8* & 7* \\ \hline
6 & 1163 & 26 & 13 & 10 & 9* \\ \hline
7 &  & 38 & 17 & 12 & 11 \\ \hline
8 &  & 86 & 21 & 16 & 13 \\ \hline
\end{tabular}
\end{center}
\caption{Known upper bounds of $f(m,k)$ for small values of $m$ and $k$. Starred values show confirmed equality.}
\label{table:smallf(m,k)}
\end{figure}

For all values that have been checked by computer, the bounds given by Theorem \ref{thrm:constructionBound} and known values for $f(k)$ are tight. However, it would be rash to conjecture that this equality continues for larger $m$ and $k$. The tournaments constructed in Proposition \ref{prop:heavyConstruction} have much freedom left in arbitrarily directed edges, like those shown as dashed lines  in Figure \ref{fig:f(3,5)}.

\section{Dice games}\label{sec:diceGamesII}



As explained in the introduction, Oskar's dice can be used to play a 3-player game where, no matter which two dice the first two players pick, the last player can choose a die that expected to beat each. This advantage is reliant on the fact that the tournament generated by Oskar's dice (shown in Figure \ref{fig:oskarDice}) is $S_2$.

This game can be extended to an arbitrary number of players while maintaining the last player's advantage. Erd\H os showed that tournaments with Sch\"{u}tte's property exist, and it is known that any tournament can be ``realized" by some set of dice, possibly with a large number of faces (\cite{mcgarvey1953theorem}, \cite{BednayBozoki}, \cite{AngelDavis}). That is, for any tournament, there is a set of dice that can be identified with the vertices of the tournament such that if vertex $v$ is directed to vertex $u$, then the die identified with $v$ is expected to roll higher than the die identified with $u$. 

Using these two ideas, one could build a set of dice to realize an $S_k$ tournament, then have $k$ opponents choose dice from this set. Since the tournament is $S_k$, there must be a die the last player can choose that is expected to beat those $k$ dice. Unfortunately, since the number of vertices needed to have an $S_k$ tournament grows rapidly, at least on the order $f(k) > (k+2)2^{k-1}-1$, so does the number of dice needed to play such a game. For instance, $f(5) \geq 111$, so one would need to carry around at least that many dice to play the game against $5$ other opponents.

Instead, we can use the extension of Sch\"{u}tte's property to sets of tournaments to reduce the number of dice needed to play these games. For example, recall that Grime's dice (shown in Figure \ref{fig:grimeDice}) can be used to play a similar game where the last player selects a die \emph{and} decided whether to roll once or twice. This rule set gives the last player an advantage precisely because the set of two tournaments in Figure \ref{fig:grimeDice} is $S_2$. 

As shown in the previous section, $S_k$ $m$-sets of tournaments exist for any $k$, and the order of these tournaments decrease as $m$ increases. If dice could be constructed to realize a given set of $S_k$ tournaments in some way, for instance with multiple rolls, those dice could be used to play a game similar to the one played with Grime's dice.

For example, suppose $\tau = \{T_r \}$ is an $S_k$ set of $m$ tournaments, and further suppose there existed a set of dice $\{D_i\}$ that could be identified with the vertices of $\tau$ such  that if $(i,j)$ is an edge in $T_r$, then $D_i$ is expected to beat $D_j$ when each die is rolled $r$ times and the results are summed. If $k$ players select dice from this set, since $\tau$ is $S_k$, there is a die $D_i$ and number of rolls $r$, outcomes represented in $T_r$, such that when rolled $r$ times, $D_i$ is expected to roll higher than all dice selected by the first $k$ players.

Buhler, Graham, and Hales \cite{Graham_dice_construction} showed that for all $n$, there exists a set of ``maximally nontransitive" $3$-sided dice $\mathcal{A}$ such that for every tournament $T$ on $n$ vertices, there are infinitely many $r$ such that $T$ is realized by rolling the set of dice $r$ times and comparing the summed results.
This result gives us a way to realize sets of tournaments using these dice, but the number of rolls needed to realize all tournaments is asymptotic to $n$. So, if a player wants to realize a particular tournament using the set of maximally nontransitive dice described in \cite{Graham_dice_construction}, they may need to roll quite a few times. 

In \cite{jeffries2024generalizations}, this author showed that sets of dice exist to realize any two tournament $T_1$ and $T_2$, such that $T_1$ shows which die is expected to win when rolled once and $T_2$ shows which die is expected to win when rolled twice and summed.
The ability to realize a general set of $m$ tournaments in the first $m$ rolls seems fairly difficult. We turn our attention to a few small examples of dice that realize $S_k$ $m$-sets of tournaments and the games they can be used to play.

The tournaments in Figure \ref{fig:f(2,2)} are an $S_2$ $2$-set of order $4$, the minimum possible order for such a set since $f(2,2) = 4$. Figure \ref{fig:f(2,2) dice} shows a set of dice that realize those two tournaments. Each arrow represents a die winning with a probability of $\frac{5}{9}$ or larger. This set of dice has the interesting property that the winner between any pair of dice switches when the dice are rolled twice. Since this set of $2$ tournaments is $S_2$, just like the tournaments generated by Grime's dice, the game played with Grime dice can also be played with this set of dice, but with fewer dice.

\begin{figure}[H] \centering 
\includegraphics[scale=1]{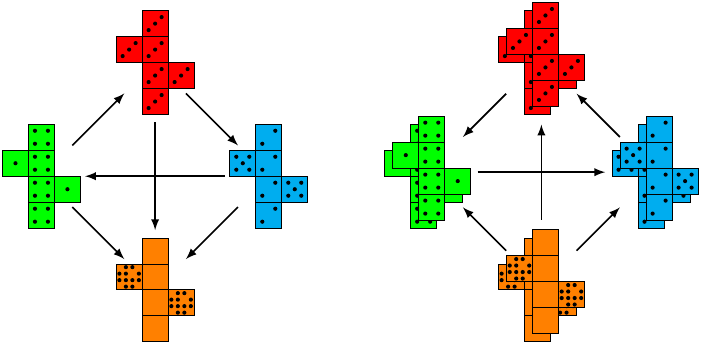}
\caption{Dice that realize the $S_2$ $2$-set of tournaments shown in Figure \ref{fig:f(2,2)}.}
\label{fig:f(2,2) dice}
\end{figure}

Figure \ref{fig:f(3,2) dice} shows another interesting set of dice. These dice have the property that when rolled once, the red die is expected to win, when rolled twice, the blue die is expected to win, and when rolled thrice, the green die is expected to win. Each bold arrow shows a winning probability of at least $\frac{7}{12}$. In a 3 player game, when 2 players select dice, the last player is forced to take the remaining die, and so they can justify getting to select how many times the dice are rolled. Since this set of tournaments is again $S_2$, the third player maintains an advantage.

The set $\{\{10, 10, 10\}, \{0, 0, 30\}, \{7, 7, 19\}, \{9, 9, 14\}, \{3, 3, 26\}\}$ is a similar set of dice where the $r^\text{th}$ die in this list is expected to roll higher than all others when rolled $r$ times and the results are summed (though some win with very low probability, the lowest probability around 0.51). The set of tournaments the outcomes these dice generate is $S_4$, so these these dice can be used to play a similar 5 player game.

\begin{figure}[H] \centering 
\includegraphics[scale=1]{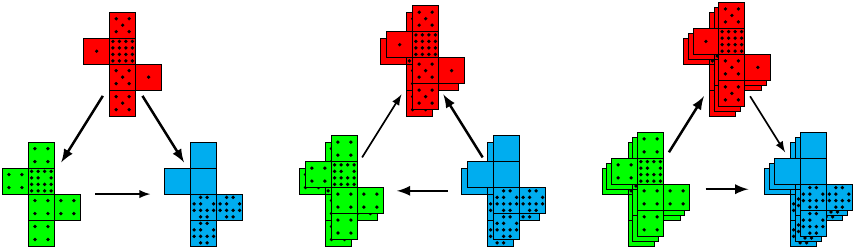}
\caption{Dice that realize a $f(3,2)$ tournament.}
\label{fig:f(3,2) dice}
\end{figure}

\section{Conclusions}\label{sec:conclusion}

This work extends Sch\"{u}tte's property to sets of tournaments motivated by dice games, but this property could could be applied to a variety situations, such as network monitoring or multi-issue voting.

There are several open questions that could extend this work. The bounds on $f(k)$ have remained un-improved for 60 years. It is likely that modern extremal techniques could be used to improve these bounds. Any improvements on the upper bound for $f(k)$ would improve the bounds for $f(m,k)$ introduced in this work. It is also likely that the bound for $f(m,k)$ given in Theorem \ref{thrm:constructionBound} could be improved, as the construction that gives this bound leaves lots of room to direct additional edges. This work also provides no non-trivial lower bound for $f(m,k)$.

On the topic of dice, the result of Buhler, Graham, and Hales \cite{Graham_dice_construction} shows that a set of dice exists to realize a set of dice with multiple rolls if one does not care about how many rolls are needed. However, it would be nice to be able to construct dice to realize a set of $m$ tournaments within the first $m$ rolls. It is possible that the classes of $S_k$ sets of tournaments generated by the construction in Proposition \ref{prop:heavyConstruction} might be easier to realize than general sets of tournaments. The sets of tournaments that come out of this construction have a particularly nice structure, where only local properties are needed in each tournament to ensure the set is $S_k$. This structure could be exploited to construct dice that realize these tournaments, which could then be used to play the unfair games described above.

\bibliographystyle{plain}
\bibliography{bibliography}

@article{erdos,
 ISSN = {00255572},
 URL = {http://www.jstor.org/stable/3613396},
 author = {P. Erdös},
 journal = {The Mathematical Gazette},
 number = {361},
 pages = {220--223},
 publisher = {Mathematical Association},
 title = {On a Problem in Graph Theory},
 urldate = {2023-01-30},
 volume = {47},
 year = {1963}
}

@article{bipartiteSchutte, 
title={Schütte's Tournament Problem and Intersecting Families of Sets}, 
volume={12}, 
DOI={10.1017/S0963548303005674}, 
number={4}, 
journal={Combinatorics, Probability and Computing}, 
publisher={Cambridge University Press}, 
author={M. Borowiecki, J. Grytczuk and M. Haluszczak and Z. TUZA}, 
year={2003}, 
pages={359–364}
}

@article{dominationirredundance,
author = {Reid, Kenneth and McRae, A.A. and Hedetniemi, S.M. and Hedetniemi, Stephen},
year = {2004},
month = {01},
pages = {},
title = {Domination and irredundance in tournaments},
volume = {29},
journal = {The Australasian Journal of Combinatorics [electronic only]}
}

@article{boundeddomination,
title = {On the bounded domination number of tournaments},
journal = {Discrete Mathematics},
volume = {220},
number = {1},
pages = {257-261},
year = {2000},
issn = {0012-365X},
doi = {https://doi.org/10.1016/S0012-365X(00)00029-7},
url = {https://www.sciencedirect.com/science/article/pii/S0012365X00000297},
author = {Xiaoyun Lu and Da-Wei Wang and C.K. Wong}
}

@misc{teaching,
  doi = {10.48550/ARXIV.2205.08357},
  url = {https://arxiv.org/abs/2205.08357},
  author = {Simon, Hans Ulrich},
  title = {Minimum Tournaments with the Strong $S_k$-Property and Implications for Teaching},
  publisher = {arXiv},
  year = {2022},
  copyright = {Creative Commons Attribution 4.0 International}
}

@article{2Szekeres,
 ISSN = {00255572},
 URL = {http://www.jstor.org/stable/3612854},
 author = {E. Szekeres and G. Szekeres},
 journal = {The Mathematical Gazette},
 number = {369},
 pages = {290--293},
 publisher = {Mathematical Association},
 title = {On a Problem of Schütte and Erdös},
 urldate = {2023-01-30},
 volume = {49},
 year = {1965}
}

@article{graham_spencer_construction, 
title={A Constructive Solution to a Tournament Problem}, 
volume={14}, 
DOI={10.4153/CMB-1971-007-1}, 
number={1}, 
journal={Canadian Mathematical Bulletin}, 
publisher={Cambridge University Press}, 
author={Graham, R. L. and Spencer, J. H.}, 
year={1971}, 
pages={45–48}
}

@phdthesis{voting,
  title={CONDORCET WINNING SETS},
  author={Von Stengel, B.},
  year={2018},
  school={London School of Economics and Political Science}
}

@article{tyszkiewiczsimple,
  title={A simple construction for tournaments with every k players beaten by a single player},
  author={Tyszkiewicz, Jerzy},
  journal={The American Mathematical Monthly},
  volume={107},
  number={1},
  pages={53},
  year={2000},
  publisher={Taylor \& Francis Ltd.}
}

@article{bozoki,
author = {Bozóki, Sándor},
year = {2014},
month = {01},
pages = {39-50},
title = {Nontransitive dice sets realizing the Paley tournaments for solving Schütte's tournament problem},
volume = {15},
journal = {Miskolc Mathematical Notes},
doi = {10.18514/MMN.2014.659}
}

@article{grime,
author = {James Grime},
title = {The Bizarre World of Nontransitive Dice: Games for Two or More Players},
journal = {The College Mathematics Journal},
volume = {48},
number = {1},
pages = {2-9},
year  = {2017},
publisher = {Taylor & Francis},
doi = {10.4169/college.math.j.48.1.2},
URL = {https://doi.org/10.4169/college.math.j.48.1.2},
eprint ={https://doi.org/10.4169/college.math.j.48.1.2}
}

@article{Graham_dice_construction,
author = {Joe Buhler and Ron Graham and Al Hales},
title = {Maximally Nontransitive Dice},
journal = {The American Mathematical Monthly},
volume = {125},
number = {5},
pages = {387-399},
year  = {2018},
publisher = {Taylor & Francis},
doi = {10.1080/00029890.2018.1427392},

URL = { 
        https://doi.org/10.1080/00029890.2018.1427392
},
eprint = { 
        https://doi.org/10.1080/00029890.2018.1427392
}
}

@article{AngelDavis,
author = {Angel, Levi and Davis, Matt},
title = {A Direct Construction of Nontransitive Dice Sets},
journal = {Journal of Combinatorial Designs},
volume = {25},
number = {11},
pages = {523-529},
doi = {https://doi.org/10.1002/jcd.21563},
url = {https://onlinelibrary.wiley.com/doi/abs/10.1002/jcd.21563},
eprint = {https://onlinelibrary.wiley.com/doi/pdf/10.1002/jcd.21563},
year = {2017}
}

@inproceedings{BednayBozoki,
          author = {D Bednay and S{\'a}ndor Boz{\'o}ki},
       booktitle = {Proceedings of the 8th Japanese-Hungarian Symposium on Discrete Mathematics and Its Applications},
          editor = {Andr{\'a}s Frank and Andr{\'a}s Recski and G{\'a}bor Wiener},
         address = {Budapest},
           title = {Constructions for nontransitive dice sets},
       publisher = {BME VIK Sz{\'a}m{\'i}t{\'a}studom{\'a}nyi {\'e}s Inform{\'a}ci{\'o}elm{\'e}leti Tansz{\'e}k},
           pages = {15--23},
            year = {2013},
             url = {https://eprints.sztaki.hu/7623/}
}

@phdthesis{jeffries2024generalizations,
  title={Generalizations of Sch{\"u}tte’s property with applications to dice games},
  type = {Doctoral Dissertation},
  author={Jeffries, Joel},
  year={2024},
  school={Iowa State University}
}

@article{mcgarvey1953theorem,
  title={A theorem on the construction of voting paradoxes},
  author={McGarvey, David C},
  journal={Econometrica: Journal of the Econometric Society},
  pages={608--610},
  year={1953},
  publisher={JSTOR}
}
\end{document}